\documentclass[10.9pt,a4paper]{article}

\usepackage[a4paper,margin=1in]{geometry}
\usepackage[T1]{fontenc}
\usepackage{microtype}

\usepackage{booktabs, multirow}
\usepackage{fancyhdr}
\usepackage{tikz}
\usetikzlibrary{arrows.meta}
\usetikzlibrary{positioning}
\usepackage{bbm}
\usepackage{natbib}
\usepackage{amstext}
\usepackage{amsthm}
\usepackage{epsfig}
\usepackage{pdfpages}
\usepackage{amsmath,amssymb,amsthm}
\usepackage{graphicx}
\usepackage{calc}
\usepackage{dsfont}
\usepackage{rotating}
\usepackage{multirow}
\usepackage{amsmath}
\usepackage{lscape}
\usepackage{tabularx}
\usepackage{bm}
\usepackage{color}
\usepackage{enumerate,array}
\usepackage{verbatim}
\usepackage{stackengine}
\usepackage{enumitem}
\usepackage{mathrsfs}
\usepackage{titling}
\usepackage{authblk}

\usepackage{hyperref}
\definecolor{cornellred}{rgb}{0.7, 0.11, 0.11}
\hypersetup{    
	colorlinks=true,
	linkcolor=blue,
	citecolor=blue,
	urlcolor=black
}

\numberwithin{equation}{section}

\allowdisplaybreaks

\newtheorem{theorem}{Theorem}[section]

\theoremstyle{definition}

\newtheorem{remark}[theorem]{Remark}

\newcommand{\R}{\mathbb{R}} 
\newcommand{\Q}{\mathbb{Q}}

\newcommand{\N}{\mathbb{N}}

\newcommand{\E}{\mathbb{E}}

\newcommand{\cB}{\mathcal{B}}
\newcommand{\cC}{\mathcal{C}}
\newcommand{\cF}{\mathcal{F}}
\renewcommand{\Pr}{\mathbb{P}}
\newcommand{\1}{\mathbbm{1}}

\renewenvironment{abstract}
  {\par\noindent\small\textbf{Abstract.}\ }
  {\par\normalsize\vspace{1em}}

\newtheorem{innercustomgeneric}{\customgenericname}
\providecommand{\customgenericname}{}
\newcommand{\newcustomtheorem}[2]{%
  \newenvironment{#1}[1]
  {%
   \renewcommand\customgenericname{#2}%
   \renewcommand\theinnercustomgeneric{##1}%
   \innercustomgeneric
  }
  {\endinnercustomgeneric}
}
\newcustomtheorem{customthm}{Assumption}
\newcustomtheorem{customcor}{Corollary}

\makeatletter
\def\singlespace{\deltaf\baselinestretch{1}\@normalsize}

\pretitle{\begin{center}\bfseries\LARGE}
\posttitle{\par\end{center}\vskip 1em}

\preauthor{\begin{center}\large}
\postauthor{\par\end{center}\vskip 0.5em}

\title{Stable convergence of partial sum processes towards discontinuous limits}

\author[1]{Johannes Brutsche}

\affil[1]{Albert-Ludwigs-Universität Freiburg}

\date{}

\begin{document}

\maketitle

\vspace{-1cm}

\begin{abstract}
We develop a stable convergence theorem for partial sum processes on sample-size dependent stochastic bases. The result allows multidimensional semimartingale limits that have conditionally independent increments and both a continuous and discontinuous martingale part. Motivated by infill asymptotics, it complements classical Gaussian stable limit theorems and supports applications to likelihood based statistical inference.
\end{abstract}

\vspace{0.5em}

\medskip 
\noindent {\bf Keywords:} 
Stable convergence on discretized filtrations \\ Multidimensional discontinuous PII \\ Partial sum process

\section{Introduction}

In this article, we derive a stable convergence result for discontinuous processes of the form $X_t^n = \sum_{k=1}^{\lfloor nt\rfloor} \chi_{nk}$, where each $X^n$ is adapted to an $n$-dependent stochastic basis $\mathcal{B}^n$. Our result is tailored to statistical applications in the context of infill asymptotics. Here, a discrete record $(Z_{k/n})_{k=0,1,\dots,n}$ of some stochastic process $Z=(Z_t)_{t\in [0,1]}$ is observed, the filtration of the basis $\mathcal{B}^n$ corresponds to the observed information and typical examples of $\chi_{nk}$ are given as $\chi_{nk}=f_k(Z_0,Z_{1/n},\dots, Z_{k/n})$ for suitable functions $f_k$. For example, suitably scaled and centered functions of the squared increments arise in many procedures for volatility inference of a semimartingale~$Z$. Another application of stable limit theorems for partial sum processes is the following: If $Z$ is a Markov process with transition density $p_t(x,y;\theta)$ from state $x$ to state $y$ within time $t$ that depends on some parameter $\theta$, we may define $\chi_{nk}(\theta)=\log(p_{1/n}(Z_{(k-1)/n},Z_{k/n};\theta))$ in which case $X^n=X^n(\theta)$ equals the sequential log-likelihood process at $\theta$. By means of Cramér--Wold, this can be extended to a finite collection $(X^n(\theta_1),\dots, X^n(\theta_M))$. Under suitable additional conditions, our stable limit result can subsequently be transferred to the process $(X^n(\theta))_\theta$, and finally to the corresponding maximum likelihood estimator (MLE), see~\cite{Brutsche/Rohde} for an application in which the limit indeed possesses jumps. 

The first result that is applicable to partial sum processes, without assuming a certain nestedness condition on the filtrations that does not hold true for infill asymptotics, is Theorem~$2.1$ in~\cite{Jacod_stableGaussian}. It establishes stable convergence towards a conditionally Gaussian process assuming a Lindeberg-type condition. Stable convergence of semimartingales towards discontinuous limits is studied in~\cite{Jacod_stablePII}, in particular in Theorem~$4.1$. However, this result in its current formulation does not treat convergence of processes that are defined on different stochastic bases $\mathcal{B}^n$. A combination of both proof techniques was used in~\cite{Brutsche/Rohde_Supp}, Proposition~H.1, to derive a result that bridges this gap. This result is formulated in one dimension on the standard Wiener space and covers limiting semimartingales without a continuous martingale part and as such was specifically tailored to the context of the main theorem on stable convergence of a certain MLE towards a Poissonian-type limit. Our result in this paper generalizes the above mentioned Proposition~H.1 to multidimensional semimartingales having limits with both continuous and discontinuous martingale parts. The assumption of the standard Wiener space is relaxed to probability spaces where the martingale representation property holds with respect to a continuous square integrable martingale (see Assumption~(A2)).

\section{Setting and assumptions}\label{Section:assumptions}

Let $\cB=(\Omega,\cF, (\cF_t)_{t\in [0,1]}, \Pr)$ be a stochastic basis that satisfies the usual conditions and $\cF=\cF_1$. We call an \textit{extension} of $\cB$ another filtered probability space $\tilde{\cB}=(\tilde{\Omega},\tilde{\cF},(\tilde{\cF}_t)_{t\in [0,1]}, \tilde{\Pr})$ that is constructed in the following way: for an auxiliary filtered space $(\hat{\Omega},\hat{\cF},(\hat{\cF}_t)_{t\in [0,1]})$, for which each $\sigma$-field $\hat{\cF}_{t-}$ is separable, and a transition probability $Q_\omega(d\hat{\omega})$ from $(\Omega,\cF)$ into $(\hat{\Omega},\hat{\cF})$, we set
\[ \tilde{\Omega} = \Omega\times\hat{\Omega}, \quad \tilde{\cF}=\cF\otimes\hat{\cF}, \quad \tilde{\cF}_t = \bigcap_{s>t} \cF_s\otimes\hat{\cF}_s, \quad \textrm{ and } \quad \tilde{\Pr}(d\omega,d\hat{\omega}) = \Pr(d\omega)Q_\omega(d\hat{\omega}).\]
Such an extension is called \textit{very good} if every $\cB$-martingale is also a $\tilde{\cB}$-martingale. We say that a process $X$ on the extension is a $\cF$-conditional process with independent increments (PII) if for $\Pr$-almost all $\omega$ the process $X(\omega,\cdot)$ is a process with independent increments on the space $\cB_\omega=(\hat{\Omega},\hat{\cF},(\hat{F}_t)_{t\in [0,1]}, Q_\omega)$. In~\cite{Jacod_stablePII}, the semimartingales that are $\cF$-conditional PIIs on a very good extension are characterized in terms of their characteristics. A sequence of random variables $(X_n)_{n\in\N}$ taking values in a metric space $E$ that is defined on $\mathcal{B}$ is said to converge $\cF$-stably to $X$ defined on the extension $\tilde{\mathcal{B}}$ if $\E[Yf(X_n)]\rightarrow \tilde{\E}[Yf(X)]$ for all bounded and continuous $f:E\rightarrow\R$ and all bounded random variables $Y$ on $(\Omega,\cF)$. This notion of convergence was introduced in~\cite{Renyi} and is slightly stronger than mere convergence in law. For an overview, the reader is referred to~\cite{Haeusler/Luschgy} and~\cite{Jacod/Shiryaev} for the particular context of stochastic processes.\\

Throughout this article, we assume the following:
\begin{itemize}
\item[(A1)] There exists a countable collection $\{V_m: m\in\N\}$ of bounded random variables which is dense in $L^2(\Omega,\cF,\Pr)$.
\item[(A2)] There exists a $p$-dimensional continuous and square integrable martingale $M$ such that the \textit{martingale representation property} holds, i.e. every local martingale $N$ on $\cB$ can be written as $N=N_0+u\bullet M$ for some predictable $\R^p$-valued process $u$.
\end{itemize}

\begin{remark}
Let $W$ be a $p$-dimensional Brownian motion and $(\cF_t)_{t\in [0,1]}$ be its completed, induced filtration. As the $\sigma$-field $\cF$ generated by $W$ up to time $T=1$ is countably generated, assumption~(A1) is satisfied by Proposition~$3.4.5$ in~\cite{Cohn}. On the other hand, $W$ is clearly square integrable and (A2) is fulfilled by Theorem~$19.11$ in~\cite{Kallenberg}. In particular, our results are applicable in most settings concerning high-frequency statistics for data that is given as a strong solution to a stochastic differential equation driven by (multidimensional) Brownian motion.
\end{remark}

We define the discretized filtration $(\cF_t^n)_{t\in [0,1]}$ via $\cF_t^n = \cF_{\lfloor nt\rfloor/n}$. Corresponding to this, we define the stochastic basis $\cB^n=(\Omega,\cF,(\cF_t^n)_{t\in [0,1]},\Pr)$ and introduce $M^n=(M_t^n)_{t\in [0,1]}$ with $M_t^n = M_{\lfloor nt\rfloor/n}$ which is a square integrable $\cB^n$-semimartingale. For each $n\in\N$, let $X^n$ be a $\cB^n$-semimartingale with
\[ X_t^n = \sum_{k=1}^{\lfloor nt\rfloor} \chi_{nk},\]
where $\chi_{nk}$ is a $\cF_{k/n}^n$-measurable and square integrable $\R^d$-valued random variable. Additionally, we assume there exists a constant $\kappa>0$ such that with $\|\cdot\|_2$ denoting the Euclidean norm
\begin{align}\label{condition:kappa}
\E\left[\sum_{k=1}^n \|\chi_{nk}\|_2^2\1_{\{\|\chi_{nk}\|_2>\kappa\}}\right] \longrightarrow\ 0\quad \textrm{ and }\quad \E\left[\sum_{k=1}^n \|M_{k/n} - M_{(k-1)/n}\|_2^2\1_{\{\|M_{k/n} - M_{(k-1)/n}\|_2>\kappa\}}\right] \longrightarrow\ 0.
\end{align}
For square integrable semimartingales $N_1,N_2$ we denote the predictable quadratic covariation of their square integrable martingale parts by $\langle N_1,N_2\rangle$ and set $\langle N_1\rangle = \langle N_1,N_1\rangle$. In case $N_k$ is $d_k$-dimensional, we denote by $N_k(i)$ the $i$-th component of $N_k$ and set $\langle N_1,N_2\rangle = (\langle N_1(i),N_2(j)\rangle)_{ij}$.
The process $(X^n,M^n)$ is a $\R^{d+p}$-dimensional semimartingale and its first characteristic $B^n$, its second modified characteristic $C^n$ and its third characteristic $\nu^n$ are given as (see Theorem~II.3.11(b) and II.3.18 in~\cite{Jacod/Shiryaev})
\[ B_t^n = \begin{pmatrix} \sum_{k=1}^{\lfloor nt\rfloor} \E[\chi_{nk}|\cF_{(k-1)/n}] \\ 0\end{pmatrix}, \quad\textrm{ and }\quad C^n = \begin{pmatrix}
\langle X^n,X^n\rangle & \langle X^n,M^n\rangle \\ \langle M^n,X^n\rangle & \langle M^n,M^n\rangle
\end{pmatrix},  \]
and for any measurable $g:\R^{d+p}\rightarrow [0,\infty)$ with $g(0)=0$,
\[ (g\star \nu^n)_t = \sum_{k=1}^{\lfloor nt\rfloor} \E\left[\left. g\left(\chi_{nk},  M_{k/n} - M_{(k-1)/n} \right)\right| \cF_{(k-1)/n}\right].\]

\section{Stable convergence result}

In this section, we provide the main theorem on the stable convergence of the process $X^n$ towards a limit that is a conditional PII with both a continuous and discontinuous component.

\begin{theorem}\label{thm:stable_conv}
Let the setup and assumptions from Section~\ref{Section:assumptions} be in place. Assume there exist the following:
\begin{itemize}
\item a continuous process $B$ with values in $\R^{d+p}$ that is of finite variation,
\item a random measure $\nu$ on $[0,1]\times\R^d\times\R^p$ that is supported on $[0,1]\times\R^d\times\{0\}$, not charging $[0,1]\times\{0\}$ and satisfying $\nu(\{t\}\times\R^d\times\R^p)=0$ identically,
\item a continuous process $C^X=(C^X(i,j))_{1\leq i,j\leq d}$ that can be decomposed as $C^X = \hat{C} + (f_{ij}\star\nu)_{1\leq i,j\leq d}$ for $f_{ij}:\R^{d+p}\rightarrow\R$ with $f_{ij}(x,y) = x_ix_j$, and $\hat{C}$ being continuous and symmetric such that $\hat{C}_t - \hat{C}_s\in \R^{d\times d}$ is positive semidefinite for $s\leq t$.
\end{itemize}
Moreover, we assume that the following convergences hold true for all $t\in [0,1]$, $1\leq i,j\leq d$ and $1\leq q\leq p$ as $n$ tends to infinity:
\begin{align}\label{condition:drift}
\sup_{s\leq t} \left\| B_s^n - B_s\right\|_2\longrightarrow_{\Pr} 0,  
\end{align}
\begin{align}\label{condition:cov_X}
\sum_{k=1}^{\lfloor nt\rfloor} \E\left[ \big(\chi_{nk}(i)-\E[\chi_{nk}(i)|\cF_{(k-1)/n}]\big)\big(\chi_{nk}(j)-\E[\chi_{nk}(j)|\cF_{(k-1)/n}]\big) | \cF_{(k-1)/n}\right]\longrightarrow_{\Pr} C_t^X(i,j),
\end{align}
\begin{align}\label{condition:cov_XM}
\sum_{k=1}^{\lfloor nt\rfloor} \E\left[ \big(\chi_{nk}(i)-\E[\chi_{nk}(i)|\cF_{(k-1)/n}]\big)\big(M_{k/n}(q)-M_{(k-1)/n}(q)\big) | \cF_{(k-1)/n}\right]\longrightarrow_{\Pr} 0,
\end{align}
and
\begin{align}\label{condition:jump_char}
\sum_{k=1}^{\lfloor nt\rfloor} \E\left[\left. g\left(\chi_{nk},  M_{k/n} - M_{(k-1)/n} \right)\right| \cF_{(k-1)/n}\right] \longrightarrow_{\Pr} (g\star\nu)_t \quad \textrm{ for all }g\in\cC,
\end{align}
where $\cC$ is a countable set of Lipschitz continuous bounded nonnegative functions on $\R^{d+p}$, vanishing in a neighborhood of $0$ and being a measure-determining class for measures not charging $0$.
Then there exist a very good extension $\tilde{\cB}$ of $\cB$ and a quasi-left continuous process $X$ on $\tilde{\mathcal{\cB}}$ which is an $\cF$-conditional PII and the pair $(X,M)$ admits the characteristics $(B,C,\nu)$, with
\[ C= \begin{pmatrix} \hat{C} & 0 \\ 0 & C^M\end{pmatrix} \]
where $C^M$ is the second characteristic of $M$, and the process $X^n$ converges stably in law to $X$ (both $X^n$ and $X$ are considered to take values in the Skorohod space $\mathcal{D}([0,1],\R^d)$).
\end{theorem}

\begin{remark}
From a probabilistic perspective, it is not necessary to assume that the limiting covariation in~\eqref{condition:cov_XM} is zero. However, in typical statistical applications, $M$ is not observable. In such a case, the limiting process $X$ cannot be constructed from the observed data in case~\eqref{condition:cov_XM} has a limit different from zero.
\end{remark}

\begin{proof}[Proof of Theorem~\ref{thm:stable_conv}] Recall the constant $\kappa$ from assumption~\eqref{condition:kappa}. We define processes $\tilde{M}^n=(\tilde{M}_t^n)_{t\in [0,1]}$ and $\tilde{X}^n = (\tilde{X}_t^n)_{t\in [0,1]}$ via
\[ \tilde{M}_t^n = \sum_{k=1}^{\lfloor nt\rfloor} \left( M_{k/n}-M_{(k-1)/n}\right)\1_{\{\|M_{k/n}-M_{(k-1)/n}\|_2\leq \kappa\}} \quad\textrm{ and }\quad \tilde{X}_t^n = \sum_{k=1}^{\lfloor nt\rfloor} \chi_{nk}\1_{\{\|\chi_{nk}\|_2\leq\kappa\}}.\]
Using~\eqref{condition:kappa}, we obtain
\[ \sup_{t\leq 1} \|X_t^n - \tilde{X}_t^n\|_2 \leq \sum_{k=1}^n \|\chi_{nk}\|_2\1_{\{\|\chi_{nk}\|_2>\kappa\}} \leq \frac{1}{\kappa}\sum_{k=1}^n \|\chi_{nk}\|_2^2 \1_{\{\|\chi_{nk}\|_2>\kappa\}} \longrightarrow_{\Pr} 0. \]
Hence, using Theorem~$3.18$(a) in~\cite{Haeusler/Luschgy}, $\cF$-stable convergence $X^n\rightarrow X$ follows once we prove $\cF$-stable convergence $\tilde{X}^n\rightarrow X$. This proof is conducted in seven steps. The general idea follows the proof of Proposition~H.1 in~\cite{Brutsche/Rohde_Supp}, the main new parts are found in (i), (ii) and (vi) which are presented in greater detail. However, all other steps are given adjusted to our setting to provide a self-contained exposition.

\begin{itemize}
\item[(i)] We first show that the convergences in~\eqref{condition:drift}-\eqref{condition:jump_char} remain valid if we replace $B^n$ by the first characteristic $\tilde{B}^n$ of $(\tilde{X}^n,\tilde{M}^n)$, $\chi_{nk}$ with $\chi_{nk}\1_{\{\|\chi_{nk}\|_2\leq\kappa\}}$ and $M$ with $\tilde{M}^n$. We obtain~\eqref{condition:drift} for $\tilde{B}^n$ directly from~\eqref{condition:kappa} and the estimate
\begin{align*}
\sup_{s\leq t} \|\tilde{B}_s^n - B^n_s\|_2 \leq \sum_{k=1}^n \E\left[\|\chi_{nk}\|_2\1_{\{\|\chi_{nk}\|_2>\kappa\}}|\cF_{(k-1)/n}\right] \leq \frac{1}{\kappa} \sum_{k=1}^n \E\left[\|\chi_{nk}\|_2^2\1_{\{\|\chi_{nk}\|_2>\kappa\}}|\cF_{(k-1)/n}\right].
\end{align*}
Denote by $\tilde{\nu}^n$ the third characteristic of $(\tilde{X}^n,\tilde{M}^n)$. Then, for any bounded $g:\R^{d+p}\rightarrow\R$,
\begin{align*}
&\left| (g\star \tilde{\nu}^n)_t-(g\star \nu^n)_t\right|\\
&\hspace{0.2cm}\leq 2\|g\|_{\sup} \sum_{k=1}^{\lfloor nt\rfloor} \E\left[\left. \1_{\{\|\chi_{nk}\|_2> \kappa\}}+\1_{\{\|M_{k/n}-M_{(k-1)/n}\|_2> \kappa\}}\right|\cF_{(k-1)/n}\right] \\
&\hspace{0.2cm}\leq \frac{2\|g\|_{\sup}}{\kappa^2} \sum_{k=1}^{\lfloor nt\rfloor} \E\left[ \left. \|\chi_{nk}\|_2^2\1_{\{\|\chi_{nk}\|_2> \kappa\}} + \|M_{k/n}-M_{(k-1)/n}\|_2^2\1_{\{\|M_{k/n}-M_{(k-1)/n}\|_2>\kappa\}}\right|\cF_{(k-1)/n}\right] ,
\end{align*}
and this upper bound converges to zero in probability by assumption~\eqref{condition:kappa}. For one-dimensional semimartingales $Y_1,\tilde{Y}_1,Y_2,\tilde{Y}_2$, we find by the Cauchy-Schwarz inequality for symmetric and positive semidefinite bilinear forms,
\begin{align*}
\left| \langle Y_1,Y_2\rangle_t - \langle \tilde{Y}_1,\tilde{Y}_2\rangle_t\right| &= \left| \langle Y_1-\tilde{Y}_1,Y_2\rangle_t + \langle \tilde{Y}_1,Y_2-\tilde{Y}_2\rangle_t\right| \leq \sqrt{\langle Y_1-\tilde{Y}_1\rangle_t \langle Y_2\rangle_t} + \sqrt{\langle Y_2-\tilde{Y}_2\rangle_t \langle \tilde{Y}_1\rangle_t}.
\end{align*}
Moreover, note that the left side in~\eqref{condition:cov_X} equals $\langle X^n(i),X^n(j)\rangle_t$ and that of~\eqref{condition:cov_XM} equals the predictable covariation $\langle X^n(i),M^n(q)\rangle_t$. Hence, the convergences~\eqref{condition:cov_X}-\eqref{condition:cov_XM} follow for $\tilde{X}^n$ and $\tilde{M}^n$ once we show for each component $1\leq i\leq d$ and $1\leq q\leq p$ that
\[ \langle X^n(i)-\tilde{X}^n(i)\rangle_t \rightarrow_\Pr 0\quad\textrm{ and }\quad \langle M^n(q)-\tilde{M}^n(q)\rangle_t \longrightarrow_\Pr 0.\]
However, the first convergence is a consequence of~\eqref{condition:kappa} and noting that
\[ \langle X^n(i)-\tilde{X}^n(i)\rangle_t \leq \sum_{k=1}^{\lfloor nt\rfloor} \E\left[\left.\chi_{nk}(i)^2\1_{\{\|\chi_{nk}\|_2>\kappa\}} \right|\cF_{(k-1)/n}\right].\]
The second one follows analogously and the assumptions~\eqref{condition:drift}-\eqref{condition:jump_char} remain valid when $B^n$ is replaced by the first characteristic of $(\tilde{X}^n,\tilde{M}^n)$, $\chi_{nk}$ with $\chi_{nk}\1_{\{\|\chi_{nk}\|_2\leq\kappa\}}$ and $M$ with $\tilde{M}^n$. In the sequel, we apply these assumptions then also to $(\tilde{X}^n,\tilde{M}^n)$ without further notice. \\
Finally, note the following: By Theorem~3.3.1(B) in~\cite{Jacod/Protter}, we obtain the convergence $\langle M^n(i),M^n(j)\rangle_t\rightarrow_{\Pr} C_t^M(i,j)$ and by the same reasoning as above, we also see $\langle \tilde{M}^n(i),\tilde{M}^n(j)\rangle_t\rightarrow_{\Pr} C_t^M(i,j)$.

\item[(ii)] We introduce $\tilde{G}^n$ and $G$ via
\[ \tilde{G}^n := \begin{pmatrix}
\langle \tilde{X}^n,\tilde{X}^n\rangle & \langle \tilde{X}^n,\tilde{M}^n\rangle \\
\langle \tilde{M}^n,\tilde{X}^n\rangle & \langle \tilde{M}^n,\tilde{M}^n\rangle
\end{pmatrix}\qquad \textrm{ and }\qquad G := \begin{pmatrix} C^X & 0 \\ 0 & C^M\end{pmatrix}.\]
By step (i), $\tilde{G}_t^n\rightarrow_{\Pr} G_t$ for all $t\in [0,1]$. Note that $\tilde{G}^n$ equals the second modified characteristic of $(\tilde{X}^n,\tilde{M}^n)$ and hence for any vector $u\in\R^{d+p}$, Proposition~II.2.17(b) in~\cite{Jacod/Shiryaev} reveals that $u^T\tilde{G}_t^n u$ is non-decreasing in $t$.\\ \pagebreak

For $v\in\R^p$, $v^TC_t^M v$ is increasing by Proposition~II.2.9 in~\cite{Jacod/Shiryaev}. For $w\in\R^d$, 
\begin{align*}
w^T(C_t^X-C_s^X) w &= w^T(\hat{C}_t - \hat{C}_s)w + w^T\left( \int_{(s,t]}\int_{\R^d}\int_{\R^p} f_{ij}(x,y) d\nu(du,dx,dy\right)_{ij} w \\
 &= w^T(\hat{C}_t - \hat{C}_s)w + \int_{(s,t]}\int_{\R^d}\int_{\R^p} (w^T x)^2 d\nu(du,dx,dy) \geq 0,
\end{align*} 
where the last step uses the assumption on $\hat{C}$. Hence, $u^T G_t u$ is also increasing for $u\in\R^{d+p}$. By a standard argument (that allows to deduce uniform convergence from pointwise convergence of non-decreasing functions to a non-decreasing limiting function), we then obtain for any $u\in\R^{d+p}$
\[ \sup_{t\leq 1} \left|u^T\tilde{G}_t^n u - u^TG_t u\right| = \sup_{t\leq 1} \left| u^T (\tilde{G}_t^n - G_t)u\right| \longrightarrow_{\Pr} 0. \]
By setting $u$ to the standard basis vectors, we conclude uniform stochastic convergence of the diagonal terms. Subsequently choosing $u=(1,1,0,\dots, 0)$ gives uniform stochastic convergence of the first off-diagonal entries. Repeating this finally allows to conclude $\sup_{t\leq 1} \| \tilde{G}_t^n - G_t\| \longrightarrow_{\Pr} 0$ for any matrix norm $\|\cdot\|$.

\item[(iii)] By the convergences~\eqref{condition:drift}-\eqref{condition:jump_char} and that of $\langle \tilde{M}^n\rangle$, the processes $B$, $G$ and $g\star\nu$ for $g\in\mathcal{C}$ are $(\cF_t)_{t\in [0,1]}$-adapted. By assumption, $B$ and $\hat{C}$ are continuous in $t$, and for $C^X$ (and hence $G$) and $g\star\nu$, continuity in $t$ follows from the assumption $\nu(\{t\}\times \R^{d+p})=0$. Therefore, $B,C$ and $g\star\nu$ are adapted and continuous processes, hence $(\cF_t)_{t\in [0,1]}$-predictable. Theorem~$3.2$ in~\cite{Jacod_stablePII} then provides existence of a very good extension $\mathcal{B}'$ of $\mathcal{B}$ and a quasi-left continuous process $X$ on $\mathcal{B}'$ which is an $\cF$-conditional PII such that the pair $(X,M)$ admits the characteristics $(B, C, \nu)$. By the same result, $X$ can be realized as follows: Let $(\hat{\Omega}, \hat{\cF}, (\hat{\cF}_t)_{t\in [0,1]})$ be the canonical space of all càdlàg functions on $[0,1]$ with values in $\R^d$ and $X$ the canonical process. Then $\Pr'(d\omega,d\hat{\omega}) = \Pr(d\omega)Q_\omega(d\hat{\omega})$ and $Q_\omega$ is entirely determined by the first $d$ coordinates of $B$, $C^X$ and $\nu(dt, dx,\{0\})$.

\item[(iv)] Recall the countable collection $\{V_m: m\in\N\}$ from Assumption~(A1) and define $N^m := (N_t^m)_{t\in [0,1]}$ via $N_t^m = \E[V_m|\cF_t]$. According to IX.7.9 and IX.7.10 in~\cite{Jacod/Shiryaev}, we have\smallskip
\begin{itemize}
\item[(A)] Every bounded martingale on $(\Omega,\cF, (\cF_t)_{t\in [0,1]},\Pr)$ is the limit in $L^2$, locally uniformly in time, of a sequence of sums of stochastic integrals w.r.t. a finite number of different $N^m$.
\item[(B)] $(\cF_t)_{t\in [0,1]}$ is the smallest filtration, up to $\Pr$-null sets, w.r.t. which all $N^m$, $m\in\N$, are adapted.
\end{itemize}
Let $m\in\N$ and introduce the processes $N(n)^m=(N(n)_t^m)_{t\in [0,1]}$ via $N(n)_t^m = N_{\lfloor nt\rfloor/n}^m$. Because $N^m$ is bounded, $\sup_{\omega,t,n}|N(n)_t^m(\omega)| <\infty$ and as $N^m$ is continuous by the martingale representation property in Assumption~(A2), we obtain convergence $(\tilde{M}^n, N(n)^1, \dots, N(n)^m) \rightarrow_\Pr( M, N^1, \dots, N^m)$ in the Skorohod space $\mathcal{D}([0,1],\R^{p+m})$. Moreover, we may consider $\mathcal{N}=(N^m)_{m\in\N}$ and $\mathcal{N}(n) = (N^m(n))_{m\in\N}$ as a process with paths in the Skorohod space $\mathcal{D}([0,1],\R^\N)$ and $\tilde{\mathcal{H}}^n =(g\star \tilde{\nu}^n)_{g\in\mathcal{C}}$, $\mathcal{H}=(g\star\nu)_{g\in\mathcal{C}}$ as processes in $\mathcal{D}([0,1],\R^\mathcal{C})$. Then, our convergence assumption and (ii) reveal convergence in probability
\begin{align}\label{eq_proof:conv_1}
\left( \tilde{B}^n, \tilde{G}^n, \tilde{\mathcal{H}}^n, \tilde{M}^n, \mathcal{N}(n)\right) \longrightarrow_{\Pr}\ \left( B,G,\mathcal{H},M, \mathcal{N}\right)
\end{align} 
in the Skorohod sense. By construction of $\tilde{X}^n$ and $\tilde{M}^n$, the jumps of $(\tilde{X}^n,\tilde{M}^n)$ are uniformly bounded, such that Theorem~VI.4.18 and Lemma~VI.4.22 in~\cite{Jacod/Shiryaev} together with (ii) reveal that $(\tilde{X}^n,\tilde{M}^n)$ is tight in $\mathcal{D}([0,1],\R^{d+p})$. Moreover, the process on the right-hand side of~\eqref{eq_proof:conv_1} is continuous and we conclude with Corollary~VI.3.33 in~\cite{Jacod/Shiryaev} that $(\tilde{X}^n, \tilde{M}^n, \tilde{B}^n, \tilde{G}^n, \tilde{\mathcal{H}}_n, \mathcal{N}(n))$ is tight in the respective Skorohod space. Moreover, for any limiting process $(\hat{X},\hat{M}, \hat{B}, \hat{G},\hat{\mathcal{H}}, \hat{\mathcal{N}})$, we find $\mathcal{L}(\hat{M}, \hat{B}, \hat{G},\hat{\mathcal{H}}, \hat{\mathcal{N}}) = \mathcal{L}(M,B,G,\mathcal{H},\mathcal{N})$.

\item[(v)] We choose a subsequence, indexed in $n_k$, such that the sequence $\mathcal{L}(\tilde{X}^{n_k}, \tilde{M}^{n_k}, \tilde{B}^{n_k}, \tilde{G}^{n_k},\tilde{\mathcal{H}}^{n_k}, \mathcal{N}(n_k))$ of distributions converges weakly to some measure $\overline{\Q}$ on the corresponding image. From what precedes, one can realize the limit as follows: Consider again the canonical space $(\hat{\Omega},\hat{\cF}, (\hat{\cF}_t)_{t\in [0,1]})$ of càdlàg functions on $[0,1]$ with values in $\R^d$ and the canonical process $X$. Then set $\tilde{\Omega} = \Omega\times \hat{\Omega}$, $\tilde{\cF}=\cF\otimes \hat{\cF}$ and $\tilde{\cF}_t = \bigcap_{s> t} \cF_s\otimes\hat{\cF}_s$. Since $\cF=\sigma(V_m: m\in\N)$ up to $\Pr$-null sets, the pullback measure of $\overline{\Q}$ is a measure on $(\tilde{\Omega},\tilde{\cF})$, in particular there exists a probability measure $\tilde{\Pr}$ on $(\tilde{\Omega},\tilde{\cF})$ whose $\Omega$-marginal is $\Pr$, and such that $\mathcal{L}(\tilde{X}^{n_k}, \tilde{M}^{n_k}, \tilde{B}^{n_k}, \tilde{G}^{n_k},\tilde{\mathcal{H}}^{n_k}, \mathcal{N}(n_k))$ converges weakly to the law of $(X,M,B,G,\mathcal{H},\mathcal{N})$ under $\tilde{\Pr}$.\\
Hence, we have an extension $\tilde{\mathcal{B}}=(\tilde{\Omega},\tilde{\cF},(\tilde{\cF}_t)_{t\in [0,1]},\tilde{\Pr})$ of $\mathcal{B}=(\Omega,\cF,(\cF_t)_{t\in [0,1]},\Pr)$ with a disintegration $\tilde{\Pr}(d\omega,d\hat{\omega})=\Pr(d\omega)\tilde{Q}_\omega(d\hat{\omega})$ (since $(\hat{\Omega},\hat{\cF})$ is Polish, see II.1.2 in~\cite{Jacod/Shiryaev}). Up to $\tilde{\Pr}$-null sets, the filtrations $(\cF)_{t\in [0,1]}$ and $(\tilde{\cF}_t)_{t\in [0,1]}$ are generated by $\mathcal{N}$ and $(\mathcal{N},X)$, respectively (this follows from property (B) in step (iv)).

\item[(vi)] In this step, we prove that $(X,M)$ is a semimartingale on the stochastic basis $\tilde{\mathcal{B}}$ with respective characteristics $(B,C,\nu)$. Let $\mu^n$ be the jump measure of $(\tilde{X}^n,\tilde{M}^n)$. Then all components of the processes $\mathcal{N}(n)$, $(\tilde{X}^n,\tilde{M}^n) - \tilde{B}^n$ and $(g\star \mu^n - g\star\tilde{\nu}^n)_{g\in\mathcal{C}}$ are $\mathbb{F}^n$-local martingales with uniformly bounded jumps. Additionally, by weak convergence of $(\tilde{X}^n,\tilde{M}^n)$ to $(X,M)$, Corollary~VI.2.8 in~\cite{Jacod/Shiryaev} provides weak convergence of $g\star\mu^n$ to $g\star \mu$ for all $g\in\mathcal{C}$. Hence, by Proposition~IX.1.17 in~\cite{Jacod/Shiryaev}, all components of 
\begin{align}\label{eq:processes_are_F-martingales}
\mathcal{N},\quad (X,M)-B \quad\textrm{ and }\quad (g\star \mu - g\star\nu)_{g\in\mathcal{C}}
\end{align}  
($\mu$ denoting the jump measure associated to $(X,M)$) are $(\tilde{\cF}_t)_{t\in [0,1]}$-local martingales. Let $\tilde{L}^n = (\tilde{X}^n,\tilde{M}^n)-\tilde{B}^n$ and $L=(X,M)-B$, which should both be read as a column vector. Consider $(\tilde{L}^n,\tilde{G}^n), (L,G)$ as elements in $\R^q$ with $q=(d+p)(d+p+1)$ and define stopping times
\begin{align*}
\tau_R^n &:= \inf\{t\in [0,1]: \|(\tilde{L}_t^n,\tilde{G}_t^n)\|_2\geq R \textrm{ or } \|(\tilde{L}_{t-}^n,\tilde{G}_{t-}^n)\|_2\geq R\}\wedge 1, \\ 
\tau_R &:= \inf\{t\in [0,1]: \|(L_t,G_t)\|_2\geq R \textrm{ or } \|(L_{t-},G_{t-})\|_2\geq R\}\wedge 1.
\end{align*} 
Following the construction in the proof of Proposition~IX.1.17 in~\cite{Jacod/Shiryaev}, it is possible to find a sequence $(R_m)_{m\in\N}$ with $R_m\nearrow\infty$ such that the weak convergence in (v) can be extended to show $\mathcal{L}((\tilde{L}^n)^{\tau_{R_m}^n},(\tilde{G}^n)^{\tau_{R_m}^n}, \tilde{X}^n, \mathcal{N}(n))$ converges weakly to $\mathcal{L}(L^{\tau_{R_m}},G^{\tau_{R_m}},X,\mathcal{N})$. Then, the $\mathbb{F}^n$-local martingale 
\[ (\tilde{L}^n)^{\tau_{R_m}^n}((\tilde{L}^n)^{\tau_{R_m}^n})^T - (\tilde{G}^n)^{\tau_{R_m}} \]
has bounded jumps and Proposition~IX.1.17 in~\cite{Jacod/Shiryaev} shows that $L^{\tau_{R_m}}(L^{\tau_{R_m}})^T-G^{\tau_{R_m}}$ is a $(\tilde{\cF}_t)_{t\in [0,1]}$-martingale as it is bounded. As $L$ and $G$ are càdlàg, $\tau_{R_m}\nearrow 1$ a.s. for $m\to\infty$, implying that $LL^T-G$ is a $(\tilde{\cF}_t)_{t\in [0,1]}$-local martingale. This was also shown to be true for the processes in~\eqref{eq:processes_are_F-martingales}, and we conclude that $(X,M)$ is a semimartingale on $\tilde{\mathcal{B}}$ with characteristics $(B,C,\nu)$. Here, we used the decomposition of $C^X$ and the fact that $\nu(\{t\}\times\R^{d+p})=0$ by assumption which allows to extract $C$ as the second characteristic from the second modified characteristic $G$, compare Proposition~II.2.17 in~\cite{Jacod/Shiryaev}. Because all elements of $\mathcal{N}$ are $\tilde{\mathcal{B}}$-martingales, property (A) of step (v) gives that every martingale on $\mathcal{B}$ is also a martingale on $\tilde{\mathcal{B}}$ and our extension is very good. Then, Theorem~$3.2$ of~\cite{Jacod_stablePII} states that the conditional $\tilde{\Pr}$-law of $X$ knowing $\mathcal{F}$ is entirely determined by $M$ and the characteristics of the pair $(X,M)$. In particular, compare (iv), we have $\tilde{Q}_\omega= Q_\omega$ for $\Pr$-almost every $\omega\in\Omega$ and the original sequence $(\tilde{X}^n, \tilde{M}^n, \tilde{B}^n, \tilde{G}^n, \tilde{\mathcal{H}}_n, \mathcal{N}(n))$ converges in distribution to $(X,M,B,G,\mathcal{H},\mathcal{N})$ defined on the basis $\tilde{\mathcal{B}}$.

\item[(vii)] It remains to prove that the convergence is indeed $\cF$-stable. This follows as in the fourth step of the proof of Theorem~$2.1$ in~\cite{Jacod_stableGaussian}, but is presented to keep the proof self-contained. From (vi) we know that the sequence $(\tilde{X}^n,\mathcal{N}(n))$ converges in law to $(X,\mathcal{N})$. In particular, if $f:\mathcal{D}([0,1],\R^d)\rightarrow\R$ is a bounded continuous function we obtain $\E[ f(\tilde{X}^n) N(n)_1^m] \rightarrow {\E}[ f(X) N_1^m]$, where $\tilde{\E}$ denotes expectation w.r.t. $\tilde{\Pr}$. Here we used that $N(n)^m$ is a component of $\mathcal{N}(n)$ that is uniformly bounded in $n$. By boundedness of $N(n)^m$ and convergence of $N(n)^m$ to $N^m$ in the Skorohod space in probability, we then deduce $N(n)_1^m \rightarrow N_1^m$ in $L^1$ such that $\E[ f(\tilde{X}^n) N_1^m] \rightarrow \tilde{\E}[ f(X)N_1^m]$. Because $\tilde{\E}[UN_1^m] = \tilde{\E}[UV_m]$ for any bounded $\tilde{\cF}$-measurable variable $U$, we deduce $\E[ f(\tilde{X}^n) V_m] \rightarrow \tilde{\E}[ f(X)V_m]$. By (A1), any bounded $\cF$-measurable random variable $V$ is the $L^2$-limit of a sequence in $\{V_m: m\in\N\}$, hence 
\[ \E\left[ f(\tilde{X}^n) V\right] \rightarrow \tilde{\E}\left[ f(X)V\right],\]
and stable convergence is proven.
\end{itemize}
\end{proof}

\bibliographystyle{plainnat}
\bibliography{bibliography}

\end{document}